\documentclass[11pt]{article}
\usepackage[numbers,sort&compress]{natbib}
\usepackage{enumerate}
\usepackage{amscd}
\usepackage{amsmath}
\usepackage{latexsym}
\usepackage{amsfonts}
\usepackage{amssymb}
\usepackage{amsthm}
\usepackage{verbatim}
\usepackage{mathrsfs}
\usepackage{enumerate}
\usepackage{hyperref}

\theoremstyle{plain}\newtheorem{definition}{Definition}[section]
\theoremstyle{definition}\newtheorem{theorem}{Theorem}[section]
\theoremstyle{plain}\newtheorem{lemma}[theorem]{Lemma}
\theoremstyle{plain}
\theoremstyle{plain}\newtheorem{prop}[theorem]{Proposition}
\theoremstyle{remark}\newtheorem{remark}{Remark}[section]
\usepackage{xcolor}

\newcommand{\Div}{\mathrm{div}\,}
\newcommand{\B}{\Big}

\newcommand{\be}{\begin{equation}}
\newcommand{\ee}{\end{equation}}
 \newcommand{\ba}{\begin{aligned}}
 \newcommand{\ea}{\end{aligned}}

\newcommand{\fbxo}{\int_{\tilde{B}_{k}}\!\!\!\!\!\!\!\!\!\! -~\,}
\newcommand{\fbxozero}{\int_{\tilde{B}_{k_{0}}}\!\!\!\!\!\!\!\!\!\!\!\! -~\,}
\newcommand{\fbxozeroo}{\int_{B_{k_{0}}}\!\!\!\!\!\!\!\!\!\!\!\! -~\,}

\newcommand{\fbx}{\int_{B_{k}}\!\!\!\!\!\!\!\!\!\! -~}

\newcommand{\fqxoo}{\iint_{\tilde{Q}_{k}} \!\!\!\!\!\!\!\!\!\!\!\!\!\!-\hspace{-0.2cm}-\hspace{-0.3cm}-\,~}

\newcommand{\fqxoth}{\iint_{\tilde{Q}_{3}} \!\!\!\!\!\!\!\!\!\!\!\!\!\!-\hspace{-0.2cm}-\hspace{-0.3cm}-\,~}

\newcommand{\fqxol}{\iint_{\tilde{Q}_{l}} \!\!\!\!\!\!\!\!\!\!\!\!\!-\hspace{-0.2cm}-\hspace{-0.3cm}-\,~}
\newcommand{\fqxolo}{\iint_{Q_{l}} \!\!\!\!\!\!\!\!\!\!\!\!\!-\hspace{-0.2cm}-\hspace{-0.3cm}-\,~}
\newcommand{\fqxolor}{\iint_{Q(r)} \!\!\!\!\!\!\!\!\!\!\!\!\!\!\!\!\!-\hspace{-0.24cm}-\hspace{-0.3cm}-\,~\,\,\,\,}
\newcommand{\fqxolorr}{\iint_{\tilde{Q}(r)} \!\!\!\!\!\!\!\!\!\!\!\!\!\!\!\!\!-\hspace{-0.24cm}-\hspace{-0.3cm}-\,~\,\,\,\,}

\providecommand{\bysame}{\leavevmode\hbox to3em{\hrulefill}\thinspace}
  \newcommand{\f}{\frac}
    
  \newcommand{\ben}{\begin{enumerate}}
   \newcommand{\een}{\end{enumerate}}

\newcommand{\ti}{\nabla}

\newcommand{\Rmnum}[1]{\expandafter\@slowromancap\romannumeral #1@}

\allowdisplaybreaks

\numberwithin{equation}{section}
\begin{document}
\title{On Wolf's regularity criterion   of   suitable weak solutions to the Navier-Stokes equations}
\author{Quanse Jiu \footnote{
School of Mathematical Sciences, Capital Normal University, Beijing 100048, PR China. Email: jiuqs@cnu.edu.cn},\,\,\,\,Yanqing Wang\footnote{ Department of Mathematics and Information Science, Zhengzhou University of Light Industry, Zhengzhou, Henan  450002,  P. R. China Email: wangyanqing20056@gmail.com}\;~ and\,
	Daoguo Zhou\footnote{
	College of Mathematics and Informatics, Henan Polytechnic University, Jiaozuo, Henan 454000, P. R. China Email:
	zhoudaoguo@gmail.com }}
\date{}
\maketitle
\begin{abstract}
In this paper, we consider the local regularity of suitable weak solutions to the 3D incompressible Navier-Stokes equations.
By means of the local pressure projection introduced by Wolf in \cite{[Wolf10],[W15]}, we present a
		$\varepsilon$-regularity criterion below of suitable weak solutions
$$
\iint_{Q(1)}|u|^{20/7}dxdt\leq \varepsilon,
$$
which gives an   improvement of  previous corresponding
results obtained
 in
Chae and    Wolf   \cite[ Arch. Ration. Mech. Anal., 225: 549-572, 2017]{[CW]}, in
Guevara and  Phuc   \cite[  Calc. Var., 56:68, 2017]{[GP]} and in Wolf \cite[ Ann. Univ. Ferrara, 61: 149-171, 2015]{[W15]}.
 \end{abstract}
\section{Introduction}
\label{intro}
\setcounter{section}{1}\setcounter{equation}{0}

We focus on   the following   incompressible Navier-Stokes equations in   three-dimensional space
	\be\left\{\ba\label{NS}
	&u_{t} -\Delta  u+ u\cdot\ti
	u+\nabla \Pi=0, ~~\Div u=0,\\
	&u|_{t=0}=u_0,
	\ea\right.\ee
	where $u $ stands for the flow  velocity field, the scalar function $\Pi$ represents the   pressure.   The
	initial  velocity $u_0$ satisfies   $\text{div}\,u_0=0$.

	In this paper, we are concerned with the regularity   of suitable weak solutions     to
	the 3D Navier-Stokes equations \eqref{NS}.
This kind of weak solutions
	obeys the local energy inequality below,
 for a.e. $t\in[-T,0]$ ,
 \begin{align}
 &\int_{\mathbb{R}^{3}} |u(x,t)|^{2} \phi(x,t) dx
 +2\int^{t}_{-T}\int_{\mathbb{R} ^{3 }}
  |\nabla u|^{2}\phi  dxds\nonumber\\ \leq&  \int^{t}_{-T }\int_{\mathbb{R}^{3}} |u|^{2}
 (\partial_{s}\phi+\Delta \phi)dxds
  + \int^{t}_{-T }
 \int_{\mathbb{R}^{3}}u\cdot\nabla\phi (|u|^{2} +2\Pi)dxds, \label{loc}
 \end{align} where non-negative function $\phi(x,s)\in C_{0}^{\infty}(\mathbb{R}^{3}\times (-T,0) )$.

 Before going further, we shall introduce some   notations utilized throughout this paper.
For $p\in [1,\,\infty]$, the notation $L^{p}((0,\,T);X)$ stands for the set of measurable functions on the interval $(0,\,T)$ with values in $X$ and $\|f(t,\cdot)\|_{X}$ belongs to $L^{p}(0,\,T)$.
  For simplicity,   we write $$\|f\| _{L^{p,q}(Q(r))}:=\|f\| _{L^{p}(-r^{2},0;L^{q}(B(r)))}~~   \text{and}~~
  \|f\| _{L^{p}(Q(r))}:=\|f\| _{L^{p,p}(Q(r))}, $$
  where $Q(r)=B(r)\times (t-r^{2}, t)$ and $ B(r)$ denotes the ball of center $x$ and radius $r$.

 Roughly speaking, the regularity of suitable weak solutions is intimately connected to $\varepsilon$-regularity criteria  	(see,
	e.g., \cite{[CKN],[RWW],[GP],[CW],[W15],[Wolf10],[Struwe],[TY],[HWZ],[Lin],[LS]}).
A well-known $\varepsilon$-regularity  criterion   is the following one with
$p=3$: there is an absolute constant $\varepsilon$ such that, if
 \be\label{k}
 \|u\|^{p}_{L^p(Q(1))}+\|\Pi\|^{p/2}_{L^{p/2}(Q(1))}<\varepsilon,\ee
 then $u$ is bounded in some neighborhood
of   point $(0, 0)$.
This was proved by  Lin in \cite{[Lin]}(see also Ladyzenskaja and  Seregin \cite{[LS]}).
 In \cite{[Kukavica]}, Kukavica proposed three questions regarding this   regularity criteria \eqref{k}
\begin{enumerate}[(1)]
		\item If this result holds for weak solutions which are not suitable.
\item It is not known if the regularity criteria holds for $p<3$ in \eqref{k}. \label{kq2}
\item If the pressure can be removed from the condition \eqref{k}.\label{kq3}
\end{enumerate}
  Recently, Guevara and Phuc \cite{[GP]} answered Kukavica's issue \eqref{kq2} via establishing following regularity criteria
		\be\label{GP}
		\|u\|_{L^{2p, 2q} (Q(1))}+\|\Pi\|_{L^{p, q}(Q(1))}< \varepsilon,
		~~~{3}/{ q}+{2}/{p}=
		7/2~~~\text{with}~1\leq q\leq2.\ee
Later, He,   Wang and   Zhou \cite{[HWZ]}
extended		Guevara and Phuc's results to \be\label{optical}\|u\|_{L^{p,q}(Q(1))}+\|\Pi\|_{L^{1}(Q(1))}<\varepsilon,~~1\leq 2/p+3/q <2, 1\leq p,\,q\leq\infty.\ee
 For the question \eqref{kq3},   Wolf introduced the local pressure projection (for the detail,
see Section \ref{sec2} ) $\mathcal{W}_{p,\Omega}:$ $W^{-1,p}(\Omega)\rightarrow W^{-1,p}(\Omega)$ $(1<p<\infty)$ for a given bounded $C^{2}(\Omega)$ domain $\Omega\subseteq \mathbb{R}^{n}$ in  \cite{[Wolf10],[W15]} and obtained a $\varepsilon$-regularity criterion below
\be \label{wolf}
\iint_{Q(1)}|u|^{3}dxdt< \varepsilon.
\ee
In addition, very recently, in \cite{[CW]},
  Wolf and Chae studied Liouville type theorems for self-similar solutions to the Navier-Stokes
equations  by proving $\varepsilon$-regularity criteria
\be \label{wolfchae}
\sup_{-1\leq t\leq0}\int_{B(1)}|u|^{q}dx < \varepsilon,~~\f32<q\leq3.
\ee
Based  Kukavica's questions and recent  progresses \eqref{GP}-\eqref{wolf}, a natural issue is   weather the regularity criteria \eqref{k} holds for $p<3$ without pressure. The goal of this paper is devoted to this. Before we state our results, we roughly mention the novelty in \cite{[Wolf10],[W15],[CW]}.
 For any ball $B(R)\subseteq \mathbb{R}^{3}$, Wolf et al. introduced the definitions
   $$\nabla\Pi_{h}=-\mathcal{W}_{p,B(R)}(u),~~ \nabla\Pi_{1}=\mathcal{W}_{p,B(R)}(\Delta u),~~\nabla\Pi_{2}=-\mathcal{W}_{p,B(R)}( u\cdot\nabla u),$$ and set $v=u+\nabla\Pi_{h}$, then,  
   the   local energy inequality reads,
for a.e. $t\in[-T,0]$ and non-negative function $\phi(x,s)\in C_{0}^{\infty}(\mathbb{R}^{3}\times (-T,0) )$,
			 \begin{align}
  &\int_{B(r)}|v|^2\phi (x,t)  d  x+ \int^{t}_{-T }\int_{B(r)}\big|\nabla v\big|^2\phi (x,t) d  x ds  \tau\nonumber\\  \leq&   \int^{t}_{-T }\int_{B(r)} | v |^2(  \Delta \phi +  \partial_{t}\phi )  d  x d s +\int^{t}_{-T }\int_{B(r)}|v|^{2}u\cdot\nabla \phi    dsds\nonumber\\
& +\int^{t}_{-T }\int_{B(r)} \phi ( u\otimes v :\nabla^{2}\Pi_{h} ) dsds   +\int^{t}_{-T }\int_{B(r)} \phi \Pi_{1}v\cdot\nabla \phi   dxds+\int^{t}_{-T }\int_{B(r)} \phi \Pi_{2}v\cdot\nabla \phi   dxds.\label{wloc1}
 \end{align}
It is worth pointing out that any usual suitable weak solutions to the Navier-Stokes system enjoys the   local energy
inequality   \eqref{wloc1}. We refer the reader to
 \cite[Appendix A, p1372]{[CW17]} for its proof.
As stated in \cite{[Wolf10],[W15],[CW]}, the advantage of  local energy
inequality
\eqref{wloc1} removed   the non-local effect of the pressure term. Based on this,
Caccioppoli type inequalities are derived in \cite{[Wolf10],[CW]}, respectively,
 \begin{align}
 &\|u\|^{2}_{L^{3,\f{18}{5}}Q(\f{1}{2})}+ \|\nabla u\|^{2}_{L^{2}(Q(\f{1}{2}))}
 \leq
  C \|  u\|^{2}_{L^{3}(Q(1))}+C\|  u\|^{3}_{L^{3}(Q(1))}.
   \label{Wolf10}\\
   \label{cw}
 &\|u\|^{2}_{L^{3,\f{18}{5}}Q(\f{1}{2})}+  \|\nabla u\|^{2}_{L^{2}(Q(\f{1}{2}))}
 \leq
  C \|  u\|^{2}_{L^{\f{3q}{2q-3},q}(Q(1))}+C\|  u\|^{\f{3q}{2q-3}}_{L^{\f{3q}{2q-3},q}(Q(1))},~~~\f32<q\leq3.
 \end{align}
Our first result   is to derive a new Caccioppoli type inequality
\begin{prop}
Assume that $u$ is a   suitable weak solutions to the Navier-Stokes equations. There holds
\be
 \|u\|^{2}_{L^{\f{20}{7},\f{15}{4}}Q(\f{1}{2})}+ \|\nabla u\|^{2}_{L^{2}(Q(\f{1}{2}))}
 \leq
  C \|  u\|^{2}_{L^{\f{20}{7}}(Q(1))}+C\|  u\|^{4}_{L^{\f{20}{7}}(Q(1))}.
  \label{wwzc}\ee
\label{the1.1}\end{prop}
This Caccioppoli type inequality allows us to obtain our main result
 \begin{theorem}\label{the1.2}
		Let  the pair $(u,  \Pi)$ be a suitable weak solution to the 3D Navier-Stokes system \eqref{NS} in $Q(1)$.
		There exists an absolute positive constant $\varepsilon$
		such that if   $u$ satisfies	\be\label{jww}\|u\|_{L^{20/7}(Q(1))} <\varepsilon,\ee
		then, $u\in L^{\infty}(Q(1/16)).$
	\end{theorem}
\begin{remark}
This theorem is an improvement of corresponding results in \eqref{GP}-\eqref{wolfchae}.
\end{remark}
We give some comments on the proof of Proposition \ref{the1.1} and Theorem  \ref{the1.2}.
Though the non-local pressure   disappears in the local energy inequality in
\eqref{wloc1}, the  velocity field $u$ losses the  kinetic energy $\|u\|_{L^{\infty,2}}.$
In contrast with works \cite{[Wolf10],[CW]}, owing to $\|u\|^{2}_{L^{3,\f{18}{5}}Q(\f{1}{2})} $
appearing in Caccioppoli type inequalities in \eqref{Wolf10}-\eqref{cw} and  without the  kinetic energy of $u$, it seems to be difficult to apply the argument used in \cite{[Wolf10],[CW],[GP],[HWZ]} directly to obtain \eqref{wwzc}. To circumvent these difficulties, first, we observe that  every nonlinear term
contain at least $v$ in the local energy inequality \eqref{wloc1}. Meanwhile, $v$ enjoys all the energy, namely, $\|v\|_{L^{\infty}L^{2}}$ and $\|v\|_{L^{2}L^{2}}$. It would be natural to absorb $v$ by the left hand of local energy inequality   \eqref{wloc1}.
Second, we establish the Caccioppoli type inequality for $ \|u\|^{2}_{L^{\f{20}{7},\f{15}{4}}Q(\f{1}{2})}+ \|\nabla u\|^{2}_{L^{2}(Q(\f{1}{2}))}$ instead of
$\|u\|^{2}_{L^{3,\f{18}{5}}Q(\f{1}{2})}+ \|\nabla u\|^{2}_{L^{2}(Q(\f{1}{2}))}$.
However, this is not enough to yield the desired result, which is completely different from that in \cite{[GP],[HWZ]}. To this end, in the spirit of \cite{[CW]}, we utilize  Caccioppoli type inequality \eqref{wloc1} and
induction arguments developed in \cite{[TY],[CKN],[CW],[RWW]} to complete the proof of Theorem \ref{the1.2}. Third,   to the knowledge of authors, all previous authors in \cite{[TY],[CKN],[CW],[RWW]}
invoked induction arguments for  $ \fqxoo |v|^{3}
\leq \varepsilon_{1}^{2/3}.$
To bound the term $\iint |v|^{2}\nabla\Pi_{h} \cdot\nabla\phi d  \tau$ in local energy inequality \eqref{wloc1} by $ \fqxoo |v|^{3}
\leq \varepsilon_{1}^{2/3},$  one needs  $u\in L^p(I;\|\cdot\|)$ with $p\geq3$, where $I$ is an time interval.
However,
from \eqref{wwzc}, we have $u\in L^p(I;\|\cdot\|)$ with $p<3$, therefore,
induction arguments with
 $ \fqxoo |v|^{3}
\leq \varepsilon_{1}^{2/3},$ seems to break down in our case. As said above, since we have all the energy of $v$, we work with
 $$ \fqxoo |v|^{\f{10}{3}}\leq \varepsilon_{1}^{2/3},$$
  in induction arguments. Finally, this
enables us to achieve the proof of Theorem \ref{the1.2}.

The remainder of this paper  is structured as follows. In section \ref{sec2}, we explain the detail of   Wolf's the local pressure projection $\mathcal{W}_{p,\Omega}$ and present the definition of local suitable weak solutions. Then, we recall some
 interior estimates of harmonic functions, an interpolation inequality, two classical  iteration lemmas and establish an auxiliary lemma utilized in induction arguments.
 The Caccioppoli type inequality \eqref{wwzc} is derived in
  Section \ref{sec3}.
 Section \ref{sec4} is devoted to the proof of Theorem \ref{the1.2}.

{\bf Notations:}
Throughout this paper, we denote
\begin{align*}
     &B(x,\mu):=\{y\in \mathbb{R}^{n}||x-y|\leq \mu\}, && B(\mu):= B(0,\mu), && \tilde{B}(\mu):=B(x_{0},\,\mu),\\
        &Q(x,t,\mu):=B(x,\,\mu)\times(t-\mu^{2\alpha}, t),  && Q(\mu):= Q(0,0,\mu), && \tilde{Q}(\mu):= Q(x_{0},t_{0},\mu),\\
          &r_{k}=2^{-k},\quad &&\tilde{B}_{k}:= \tilde{B}(r_{k}), \quad ~~ &&\tilde{Q}_{k}:=\tilde{Q}(r_{k}).
\end{align*}
  Denote
  the average of $f$ on the set $\Omega$ by
  $\overline{f}_{\Omega}$. For convenience,
  $\overline{f}_{r}$ represents  $\overline{f}_{B(r)}$ and $\overline{\Pi}_{\tilde{B}_{k}}$
  is denoted by $\tilde{\Pi}_{k}$.
  $|\Omega|$ represents the Lebesgue measure of the set $\Omega$. We will use the summation convention on repeated indices.
 $C$ is an absolute constant which may be different from line to line unless otherwise stated in this paper.

\section{Preliminaries}\label{sec2}

We begin with  Wolf's the local pressure projection $\mathcal{W}_{p,\Omega}:$ $W^{-1,p}(\Omega)\rightarrow W^{-1,p}(\Omega)$ $(1<p<\infty)$.
 More precisely, for any  $f\in W^{-1,p}(\Omega)$, we define $W^{-1,p}(f)= \nabla\Pi$, where $\Pi$ satisfies \eqref{GMS}.
Let $\Omega$  be a  bounded domain with $\partial\Omega\in C^{1}$.
According to the $L^p$ theorem of Stokes system in \cite[Theorem 2.1, p149]{[GSS]},
there exists a unique pair $(u,\Pi)\in W^{1,p}(\Omega)\times L^{p}(\Omega)$ such that
\be\label{GMS}
-\Delta u+\nabla\Pi=f,~~ \text{div}\,u=0, ~~u|_{\partial\Omega}=0,~~ \int_{\Omega}\Pi dx=0.
\ee
Moreover, this pair is subject to the inequality
$$
\|u\|_{W^{1,q}(\Omega)}+\|\Pi\|_{L^q(\Omega)}\leq C\|f\|_{W^{-1,q}(\Omega)}.
$$
Let $\nabla\Pi= \mathcal{W}_{p,\Omega}(f)$ $(f\in L^p(\Omega))$, then $\|  \Pi\|_{L^p(\Omega)}\leq C\|f\|_{L^p(\Omega)},$ where we used the fact that $L^{p}(\Omega)\hookrightarrow W^{-1,p}(\Omega)$.  Moreover, from $\Delta \Pi=\text{div}\,f$, we see that $\|  \nabla\Pi\|_{L^p(\Omega)}\leq C(\|f\|_{L^p(\Omega)}+ \|  \Pi\|_{L^p(\Omega)}) \leq C\|f\|_{L^p(\Omega)}.$
Now, we  present the definition of suitable weak solutions of Navier-Stokes equations \eqref{NS}.
	\begin{definition}\label{defi}
		A  pair   $(u, \,\Pi)$  is called a suitable weak solution to the Navier-Stokes equations \eqref{NS} provided the following conditions are satisfied,
		\begin{enumerate}[(1)]
			\item $u \in L^{\infty}(-T,\,0;\,L^{2}(\mathbb{R}^{3}))\cap L^{2}(-T,\,0;\,\dot{H}^{1}(\mathbb{R}^{3})),\,\Pi\in
			L^{3/2}(-T,\,0;L^{3/2}(\mathbb{R}^{3}));$\label{SWS1}
			\item$(u, ~\Pi)$~solves (\ref{NS}) in $\mathbb{R}^{3}\times (-T,\,0) $ in the sense of distributions;\label{SWS2}
			\item The local energy inequality \eqref{wloc1} is valid. In addition, $ \nabla\Pi_{h}, \nabla\Pi_{1}$ and $\nabla\Pi_{2}$ meet the following fact
	\begin{align}   &\|\nabla\Pi_{h}\|_{L^p(B(R))}\leq  \|u\|_{L^p(B(R))}, \label{ph}\\
 &\|\nabla\Pi_{1}\|_{L^2(B(R))}\leq  \|\nabla u\|_{L^2(B(R))},\label{p1}\\
 &\|\nabla\Pi_{2}\|_{L^{p/2}(B(R))}\leq  \| |u|^{2}\|_{L^{p/2}(B(R))}.\label{p2}
\end{align}	\end{enumerate}
	\end{definition}
	\noindent
We list some
interior estimates
of  harmonic functions $\Delta h=0$, which will be frequently utilized later. Let $1\leq p,q\leq\infty$ and $p<r<\rho$, then, it holds
\be\label{h1}\|\nabla^{k}h\|_{L^{q}
(B(r))}\leq \f{Cr^{\f{n}{q}}}{(\rho-r)^{\f{n}{p}+k}}\|h\|_{L^{p}(B(\rho))}.\ee
\be\label{h2}
 \| h-\overline{h}_{r}\|_{L^{q}
(B(r))}\leq \f{Cr^{\f{n}{q}+1}}{(\rho-r)^{\f{n}{q}+1 }}\|h-\overline{h}_{\rho}\|_{L^{q}(B(\rho))}.\ee
 The proof of \eqref{h1} rests on the mean value property of harmonic functions. This together with
  mean value theorem leads to \eqref{h2}. We leave the detail to the reader.
 For reader's convenience, we recall an interpolation inequality.
For each $2\leq l\leq\infty$ and $2\leq k\leq6$ satisfying $\f{2}{l}+\f{3}{k}=\f{3}{2}$, according to the H\"{o}lder   inequality  and the Young  inequality, we know that
\begin{align}
\|u\|_{L^{ k,l}(Q(\mu))}&\leq C\|u\|_{L^{2,\infty}(Q(\mu))}^{1-\f {2} {l}}\|u\|_{L^{6,2}(Q(\mu))}^{\f {2} {l}}\nonumber\\
&\leq C\|u\|_{L^{2,\infty} (Q(\mu))}^{1-\f {2} {l}}(\|u\|_{L^{2,\infty} (Q(\mu))}
+\|\nabla u\|_{L^2(Q(\mu))})^{\f {2} {l}}\nonumber\\
&\leq C (\|u\|_{L^{2,\infty} (Q(\mu))}
+\|\nabla u\|_{L^2(Q(\mu))}).\label{sampleinterplation}
\end{align}
In additon, we recall two well-known iteration lemmas.		
	\begin{lemma}\cite[Lemma 2.1,   p.86 ]{[Giaquinta]}\label{iter2}
Let $\phi(t)$ be a nonegative and nondecreasing functions on [0,R]. Suppose that
$$\phi(\rho)\leq A\B[\B(\f{\rho}{r}\B)^{\alpha}+\varepsilon\B]\phi(r)+Br^{\beta}$$
for any $0<\rho\leq r\leq R$, with $A,B,\alpha,\beta$ nonnegative constants and $\beta<\alpha$. Then for any $\gamma\in(\beta,\alpha)$, there exists a constant $\varepsilon_{0}$ such that if $\varepsilon<\varepsilon_{0}$ we have for all $0<\rho\leq r\leq R$
$$ \phi(\rho)\leq C\B\{\B(\f{\rho}{r}\B)^{\beta}\phi(r)+B \rho^{\beta}\B\}.$$
where $c$ is a positive constant depending on $A,\alpha,\beta,\gamma.$
\end{lemma}
\begin{lemma}\label{iter1}\cite[Lemma V.3.1,   p.161 ]{[Giaquinta]}
			Let $I(s)$ be a bounded nonnegative function in the interval $[r, R]$. Assume that for every $\sigma, \rho\in [r, R]$ and  $\sigma<\rho$ we have			$$I(\sigma)\leq A_{1}(\rho-\sigma)^{-\alpha_{1}} +A_{2}(\rho-\sigma)^{-\alpha_{2}} +A_{3}+ \ell I(\rho)$$
			for some non-negative constants  $A_{1}, A_{2}, A_{3}$, non-negative exponents $\alpha_{1}\geq\alpha_{2}$ and a parameter $\ell\in [0,1)$. Then there holds
			$$I(r)\leq c(\alpha_{1}, \ell) [A_{1}(R-r)^{-\alpha_{1}} +A_{2}(R-r)^{-\alpha_{2}} +A_{3}].$$
		\end{lemma}

The following lemma is  motivated by \cite[Lemma  2.9,   p.558 ]{[CW]}.
  \begin{lemma}\label{CW}
Let $f\in L^{q}(Q(1))$ with $q>1$ and $0<r_{0}<1$. Suppose that   for all $(x_{0},t_{0})\in Q(1/2)$ and $r_{0}\leq r\leq \f12$
\be\label{lem2.31}\iint_{\tilde{Q}(r)}|f-\overline{f}_{\tilde{B}(r)}|^{q}\leq C r^{4}.\ee
Let $\nabla\Pi=\mathcal{W}_{q,\tilde{B}(1)} (\nabla\cdot f)$. Then for all $(x_{0},t_{0})\in Q(1/2)$ and
$r_{0}\leq r \leq\f14$, it holds
$$\iint_{\tilde{Q}(r)}|\Pi-\overline{\Pi}_{\tilde{B}(r)}|^{q}\leq C r^{4}$$
\end{lemma}
\begin{proof}
From the definition of  pressure projection $\mathcal{W}_{q,B(1)}$, we know that
\be\label{2.8}
\|\Pi\|_{L^{q}(B(1))}\leq C\|f-\overline{f}_{B(1)}\|_{L^{q}(B(1))}.
\ee
Let $\phi(x)=1, x\in \tilde{B}(\f{3r}{4}), \phi(x)=0, x\in \tilde{B}^{c}(r).$\\
Note that $$\Delta \Pi=\text{div}\mathcal{W}_{q,B(1)}^{\ast}(\nabla\cdot f).$$
 We set $\Pi=\Pi_{(1)}+\Pi_{(2)}$, where
  $$\Delta \Pi_{(1)}=-\text{div}\mathcal{W}_{q,B(1)}^{\ast}(\nabla\cdot [\phi( f-\overline{f}_{\tilde{B}(r)})]),$$
  therefore, as a consequence, it holds
  $$\Delta \Pi_{(2)}=0, x\in \tilde{B}(3r/4).$$
In view of  classical Calder\'on-Zygmund theorem, we have
\be\label{cz1}
\|\Pi_{(1)}-\overline{\Pi_{(1)}}_{\tilde{B}(r)}\|_{L^{q}(\tilde{B}(r))}\leq C\|f-\overline{f}_{\tilde{B}(r)}\|_{L^{q}(\tilde{B}(r))}.
\ee
Combining this and  hypothesis \eqref{lem2.31}, we get
$$
\|\Pi_{(1)}-\overline{\Pi_{(1)}}_{\tilde{B}(r)}\|_{L^{q}(\tilde{Q}(r))}\leq Cr^{\f{4}{q}}.
$$
 The interior estimates of harmonic functions \eqref{h2} and the triangle inequality guarantee that, for $ \theta<1/2$,
$$\ba
&\int_{\tilde{B}(\theta r)}|\Pi_{(2)}-\overline{\Pi_{(2)}}_{\tilde{B}(\theta r)}|^{q}dx\\
\leq& \f{C( r\theta )^{3+q}}{(\f{r}{2})^{3+q}}
\int_{\tilde{B}(  r/2)}|\Pi_{(2)}-\overline{\Pi_{(2)}}_{\tilde{B}(r/2)}|^{q}dx\\
\leq& C\theta^{3+q}
\int_{\tilde{B}(  r/2)}|\Pi-\overline{\Pi}_{ \tilde{B}( r/2)}|^{q}dx+
\int_{\tilde{B}(  r/2)}|\Pi_{(1)}-\overline{\Pi_{(1)}}_{\tilde{B}(r/2)}|^{q}dx.
\ea$$
This and  \eqref{cz1}   imply
$$\ba
\iint_{\tilde{Q}(\theta r)}|\Pi_{(2)}-\overline{\Pi_{(2)}}_{\tilde{B}(\theta r)}|^{q}dxds\leq C\theta^{3+q}
\iint_{\tilde{Q}(  r/2)}|\Pi-\overline{\Pi}_{ \tilde{B}( r/2)}|^{q}dxds+C r^{4}.
\ea$$
Utilizing the triangle inequality again, \eqref{cz1} and the last inequality, we have
$$\ba&\iint_{\tilde{Q}(\theta r)}|\Pi-\overline{\Pi}_{\tilde{B}(\theta r)}|^{q}dxds\\\leq&
\iint_{\tilde{Q}(\theta r)}|\Pi_{(1)}-\overline{\Pi_{(1)}}_{\tilde{B}(\theta r)}|^{q}dxds
+\iint_{\tilde{Q}(\theta r)}|\Pi_{(2)}-\overline{\Pi_{(2)}}_{\tilde{B}(\theta r)}|^{q}dxds\\
\leq & \iint_{\tilde{Q}( r)}|f-\overline{f}_{\tilde{B}(r)}|^{q}dxds+C\theta^{3+q}
\iint_{\tilde{Q}(  r/2)}|\Pi-\overline{\Pi}_{  \tilde{B} (r/2)}|^{q}dxds+C r^{4}\\
\leq&  C\theta^{3+q}
\iint_{\tilde{Q}(  r)}|\Pi-\overline{\Pi}_{ \tilde{B} (r)}|^{q}dxds+C r^{4},
\ea$$
where we used the fact that $\|g-\overline{g}_{B(r)}\|_{L^p(B(r))}\leq C\|g-c\|_{L^p(B(r))}$ with $p\geq1$.

Now, invoking Lemma \ref{iter2} and \eqref{2.8}, we see that
 $$\ba\iint_{\tilde{Q}( r)}|\Pi-\overline{\Pi}_{\tilde{B}(r)}|^{q}dx
&\leq  Cr^{4}\iint_{\tilde{Q}(1/4)}
|\Pi-\overline{\Pi}_{\tilde{B}(1/4)} |^{q}dx+Cr^{4}\\
&\leq  Cr^{4}\iint_{Q(1)}
|f-\overline{f}_{B(1)} |^{q}dx+Cr^{4}
\\&\leq Cr^{4}.
\ea$$
This completes the proof of this lemma.
\end{proof}
\section{Proof  of Proposition  \ref{the1.1}}
		\label{sec3}
		\setcounter{section}{3}\setcounter{equation}{0}

This section contains the proof of Proposition  \ref{the1.1}.
Proposition  \ref{the1.1} turns out to be a corollary of the following proposition.
 \begin{prop}\label{lebp}   Suppose that $(u,\Pi)$ is a suitable weak solution
to the Navier-Stokes equations in  $Q(R)$. Then there holds, for any $R>0$,
\be\ba \label{key ineq}
&\|u\|^{2}_{L^{\f{20}{7},\f{15}{4}}Q(\f{R}{2})}+ \|\nabla u\|^{2}_{L^{2}(Q(\f{R}{2}))}
 \leq
  CR^{-\f{1}{2}}\|\nabla u\|^{2}_{L^{\f{20}{7}}(Q(R))}
  +CR^{-2}\|\nabla u\|^{4}_{L^{\f{20}{7}}(Q(R))}.
\ea\ee
\end{prop}
\begin{proof}Consider $0<R/2\leq r<\f{3r+\rho}{4}<\f{r+\rho}{2}<\rho\leq R$. Let $\phi(x,t)$ be non-negative smooth function supported in $Q(\f{r+\rho}{2})$ such that
$\phi(x,t)\equiv1$ on $Q(\f{3r+\rho}{4})$,
$|\nabla \phi| \leq  C/(\rho-r) $ and $
|\nabla^{2}\phi|+|\partial_{t}\phi|\leq  C/(\rho-r)^{2} .$

Let $\nabla\Pi_{h}=\mathcal{W}_{20/7,B(\rho)}(u)$, then, there holds
\begin{align}
&\|\nabla \Pi_{h}\|_{L^{20/7}(Q(\rho))}\leq C\|u\|_{L^{20/7}(Q(\rho))},\label{wp1}\\
 &\|  \Pi_{1}\|_{L^{2}(Q(\rho))}\leq C\|\nabla u\|_{L^{2}(Q(\rho))},\label{wp2}\\
 &\|  \Pi_{2}\|_{L^{\f{10}{7}}(Q(\rho))}\leq C\|  |u|^{2}\|_{L^{\f{10}{7}}(Q(\rho))}.\label{wp3}
 \end{align}
Thanks to $v=u+\nabla\Pi_{h}$, the H\"older inequality and \eqref{wp1}, we arrive at
\begin{align}
 \iint_{Q(\rho)} | v_{_{B}}|^2\Big|  \Delta \phi^{4}+  \partial_{t}\phi^{4}\Big|    \leq& \f{C}{(\rho-r)^{2}}\iint_{Q(\f{r+\rho}{2})} |u|^{2}+|\nabla\Pi_{h}|^{2}
\nonumber\\\leq& \f{C\rho^{3/2}}{(\rho-r)^{2}}\B(\iint_{Q(\f{r+\rho}{2})} |u|^{\f{20}{7}}+|\nabla\Pi_{h}|^{\f{20}{7}}\B)^{\f{7}{10}}
\nonumber\\\leq& \f{C\rho^{3/2}}{(\rho-r)^{2}}\B(\iint_{Q(\rho)} |u|^{\f{20}{7}}\B)^{\f{7}{10}}.\label{3.2} \end{align}
 Combining H\"older's inequality with interpolation inequality \eqref{sampleinterplation} and Young's inequality yields
 \begin{align}
 &\iint_{Q(\rho)}|v|^{2}\phi^{3}u\cdot\nabla(\phi )  d  \tau\nonumber\\
\leq& \f{C}{(\rho-r)}\| v\phi^{2}\|_{L^{10/3}(Q(\f{r+\rho}{2}))}
\| v \|_{L^{20/7}(Q(\f{r+\rho}{2}))}\| u \|_{L^{20/7}(Q(\f{r+\rho}{2}))}\nonumber\\
\leq &\f{1}{16}\| v\phi^{2}\|^{2}_{L^{10/3}(Q(\f{r+\rho}{2}))}+\f{C}{(\rho-r)^{2}}
\| v \|^{2}_{L^{20/7}(Q(\f{r+\rho}{2}))}\| u \|^{2}_{L^{20/7}(Q(\f{r+\rho}{2}))}\nonumber\\
\leq& \f{1}{16} \B(\|v\phi^{2}\|_{L^{2,\infty}(Q(\rho))}^{2}+\|\nabla (v\phi^{2})\|_{L^{2}(Q(\rho))}^{2}\B) +\f{C}{(\rho-r)^{2}}
 \| u \|^{4}_{L^{20/7}(Q(\rho))}.
 \end{align}
By virtue of interior estimate of harmonic function \eqref{h1} and \eqref{wp1}, we conclude that
$$\ba\|\nabla^{2}\Pi_{h} \|_{L^{20/7}(Q(\f{r+\rho}{2}))}&\leq
\f{ (r+\rho) ^{ \f{21}{20}}}{(\rho-r)^{ \f{41}{20}}}
\|\nabla\Pi_{h} \|_{L^{20/7}(Q( \rho ))}\\
&\leq
\f{C \rho ^{\f{21}{20}}}{(\rho-r)^{ \f{41}{20}}}
\|u \|_{L^{20/7}(Q( \rho ))},
\ea$$
which leads to
\begin{align}
&\iint_{Q(\rho)} \phi^{4}( u\otimes v :\nabla^{2}\Pi_{h} )  \nonumber\\
\leq&
\| v\phi^{2}\|_{L^{10/3}(Q(\f{r+\rho}{2}))}
\| u \|_{L^{20/7}(Q(\f{r+\rho}{2}))}\|\nabla^{2}\Pi_{h} \|_{L^{20/7}(Q(\f{r+\rho}{2}))}
\nonumber\\
\leq &\f{1}{16} \B(\|v\phi^{2}\|_{L^{2,\infty}(Q(\rho))}^{2}+\|\nabla (\phi^{2}v)\|_{L^{2}(Q(\rho))}^{2}\B) +\f{C\rho^{\f{21}{10}}}{(\rho-r)^{\f{41}{10}}}
 \| u \|^{4}_{L^{20/7}(Q(\rho)}.
\end{align}
In light  of H\"older inequality, \eqref{wp2} and Young's inequality, we deduce that
\begin{align}
\iint_{Q(\rho)} \phi^{3} \Pi_{1}v\cdot\nabla \phi  d  \tau
&\leq \f{C}{(\rho-r)}\| v\|_{L^{2}(Q(\f{r+\rho}{2}))}
\| \Pi_{1} \|_{L^{2}(Q(\f{r+\rho}{2}))} \nonumber\\
&\leq \f{C}{(\rho-r)^{2}}\| v\|^{2}_{L^{2}(Q(\f{r+\rho}{2}))}
+\f{1}{16}\| \Pi_{1} \|^{2}_{L^{2}(Q(\rho))} \nonumber\\
&\leq\f{C\rho^{3/2}}{(\rho-r)^{2}}\B(\iint_{Q(\rho)} |u|^{\f{20}{7}}\B)^{\f{7}{10}}+\f{1}{16}\| \nabla u\|^{2}_{L^{2}(Q(\rho))}.
\end{align}
We derive from the H\"older inequality, \eqref{p2} and Young's inequality that
\begin{align}
\iint_{Q(\rho)} \phi^{3} \Pi_{2}v\cdot\nabla \phi   d  \tau
&\leq \f{C}{(\rho-r)}\| v\phi^{2}\|_{L^{\f{10}{3}}(Q(\f{r+\rho}{2}))}
\| \Pi_{2} \|_{L^{\f{10}{7}}(Q(\f{r+\rho}{2}))}\nonumber \\
&\leq \f{1}{16}\| v\|^{2}_{L^{ \f{10}{3}}(Q(\f{r+\rho}{2}))}
+\f{C}{(\rho-r)^{2}}\| \Pi_{2} \|^{2}_{L^{\f{10}{7}}(Q(\rho))} \nonumber\\&\leq \f{1}{16} \B(\|v\phi^{2}\|_{L^{2,\infty}(Q(\rho))}^{2}+\|\nabla (\phi^{2}v)\|_{L^{2}(Q(\rho))}^{2}\B) +\f{C}{(\rho-r)^{2}}
 \| u \|^{4}_{L^{20/7}(Q(\rho)}.\label{locp5}
\end{align}
The Cauchy-Schwarz inequality and \eqref{3.2} allows us to obtain that
\begin{align}
 \iint_{Q(\rho)}|\nabla( v \phi^{2})|^{2} dxds
\leq&2\B( \iint_{Q(\rho)}|\nabla    v|^{2}\phi^{4} dxds
+4 \iint_{Q(\rho)}|\nabla\phi|^{2}|v|^{2}\phi^{2}dxds\B) \nonumber\\
\leq&2  \iint_{Q(\rho)}|\nabla    v|^{2}\phi^{4} dxds
+
\f{C\rho^{3/2}}{(\rho-r)^{2}}\B(\iint_{Q(\rho)} |u|^{\f{20}{7}}\B)^{\f{7}{10}}.\label{cz}
\end{align}
Substituting \eqref{3.2}-\eqref{locp5} into \eqref{wloc1} and using \eqref{cz}, we infer that
\begin{align}
  &\sup_{-\rho^{2}\leq t\leq0}\int_{B(\rho)}|v\phi^{2}|^2   d  x+ \iint_{Q(\rho)}\big|\nabla( v\phi^{2})\big|^2  d  x d  \tau\nonumber\\\leq& \f{1}{4} \B(\|v\phi^{2}\|_{L^{2,\infty}(Q(\rho))}^{2}+\|\nabla (v\phi^{2})\|_{L^{2}(Q(\rho))}^{2}\B) +\B\{\f{C}{(\rho-r)^{2}}+\f{C\rho^{\f{21}{10}}}{(\rho-r)^{\f{41}{10}}}\B\}
 \| u \|^{4}_{L^{20/7}(Q(\rho))}\nonumber\\&
 +\f{C\rho^{3/2}}{(\rho-r)^{2}}\| u \|^{2}_{L^{20/7}(Q(\rho))}+\f{1}{16}\| \nabla u\|^{2}_{L^{2}(Q(\rho))},\nonumber
  \end{align}
  that is,
  \begin{align}
  &\sup_{-\rho^{2}\leq t\leq0}\int_{B(\rho)}|v\phi^{2}|^2   d  x+ \iint_{Q(\rho)}\big|\nabla( v\phi^{2})\big|^2  d  x d  \tau\nonumber\\\leq&   \B\{\f{C}{(\rho-r)^{2}}+\f{C\rho^{\f{21}{10}}}{(\rho-r)^{\f{41}{10}}}\B\}
 \| u \|^{4}_{L^{20/7}(Q(\rho))}\nonumber\\&
 +\f{C\rho^{3/2}}{(\rho-r)^{2}}\| u \|^{2}_{L^{20/7}(Q(\rho))}+\f{1}{16}\| \nabla u\|^{2}_{L^{2}(Q(\rho))},\label{keyl}
  \end{align}
  Together with interior estimate of harmonic function
  \eqref{h1} and \eqref{wp1} implies that
$$\ba\|\nabla\Pi_{h}\|^{2}_{L^{\f{20}{7},\f{15}{4}}Q(r)}&\leq \f{Cr^{\f{8}{5}}}{(\rho-r)^{\f{21}{10}}}\|\nabla\Pi_{h}\|^{2}_{L^{\f{20}{7}}Q(\rho)}\\
&\leq\f{Cr^{\f{8}{5} }}{(\rho-r)^{ \f{21}{10}}}\|u\|^{2}_{L^{\f{20}{7}}
L^{\f{20}{7}}(Q(\rho))}.
\ea$$
With the help of the triangle inequality, interpolation inequality  \eqref{sampleinterplation} and the last inequality, we get
$$\ba
 \|u\|^{2}_{L^{\f{20}{7},\f{15}{4}}(Q(r))} \leq& \|v\|^{2}_{L^{\f{20}{7},\f{15}{4}}(Q(r))}+\|\nabla\Pi_{h}\|^{2}_{L^{\f{20}{7},\f{15}{4}}(Q(r))}\\
\leq& C\B\{\|v\|_{L^{2,\infty}(Q(r))}^{2}+\|\nabla v\|_{L^{2}(Q(r))}^{2}\B\}+\f{r^{\f{8}{5} }}{(\rho-r)^{ \f{21}{10}}}\|u\|^{2}_{L^{\f{20}{7}}
L^{\f{20}{7}}(Q(\rho))}\\
\leq & \B\{\f{C}{(\rho-r)^{2}}+\f{C\rho^{\f{21}{10}}}{(\rho-r)^{\f{41}{10}}}\B\}
 \| u \|^{4}_{L^{20/7}(Q(\f{r+\rho}{2}))}\\&+\B\{\f{C\rho^{3/2}}{(\rho-r)^{2}}+\f{Cr^{\f{8}{5} }}{(\rho-r)^{ \f{21}{10}}}\B\}\| u \|^{2}_{L^{20/7}(Q(\rho))}+\f{1}{16}\| \nabla u\|^{2}_{L^{2}(Q(\rho))}.
\ea$$
Employing \eqref{h1} and \eqref{p1} once again, we have the estimate
$$
 \|\nabla^{2} \Pi_{h}\|^{2}_{L^{2}(Q(r))}\leq \f{Cr^{3}}{(\rho-r)^{3+2\cdot1}} \|\nabla\Pi_{h}\|^{2}_{L^{2}(Q(\f{r+\rho}{2}))}\leq \f{Cr^{3}\rho^{3/2}}{(\rho-r)^{3+2\cdot1}} \| u \|^{2}_{L^{20/7}(Q(\rho))}.
$$
This together with the triangle inequality and \eqref{keyl} leads to
\begin{align}
  \|\nabla u\|^{2}_{L^{2}(Q(r))}&\leq \|\nabla v\|^{2}_{L^{2}(Q(r))}+
 \|\nabla^{2} \Pi_{h}\|^{2}_{L^{2}(Q(r))}\nonumber\\\leq&
  \B\{\f{C}{(\rho-r)^{2}}+\f{C\rho^{\f{21}{10}}}{(\rho-r)^{\f{41}{10}}}\B\}
 \| u \|^{4}_{L^{20/7}(Q(\rho))}\nonumber\\&+ \B\{\f{C\rho^{3/2}}{(\rho-r)^{2}}+\f{Cr^{3}\rho^{3/2}}{(\rho-r)^{3+2\cdot1}} \B\} \| u \|^{2}_{L^{20/7}(Q(\rho))}+\f{1}{16}\| \nabla u\|^{2}_{L^{2}(Q(\rho))}.
\end{align}
 Eventually, we infer that
  $$\ba
&\|u\|^{2}_{L^{\f{20}{7},\f{15}{4}}(Q(r))}+ \|\nabla u\|^{2}_{L^{2}(Q(r))}
\\\leq&
  \B\{\f{C}{(\rho-r)^{2}}+\f{C\rho^{\f{21}{10}}}{(\rho-r)^{\f{41}{10}}}\B\}
 \| u \|^{4}_{L^{20/7}(Q(\rho))}\\&+\B\{\f{C\rho^{3/2}}{(\rho-r)^{2}}+\f{Cr^{3}\rho^{3/2}}{(\rho-r)^{3+2\cdot1}} +\f{Cr^{\f{8}{5} }}{(\rho-r)^{ \f{21}{10}}}\B\}\| u \|^{2}_{L^{20/7}(Q(\rho))}+\f{3}{16}\| \nabla u\|^{2}_{L^{2}(Q(\rho))}.
 \ea $$
Now, we are in a position to apply lemma \ref{iter1} to the latter to find that
$$\|u\|^{2}_{L^{\f{20}{7},\f{15}{4}}(Q(\f{R}{2}))}+ \|\nabla u\|^{2}_{L^{2}(Q(\f{R}{2}))}\leq
  CR^{-\f{1}{2}}\|\nabla u\|^{2}_{L^{\f{20}{7}}(Q(R))}
  +CR^{-2}\|\nabla u\|^{4}_{L^{\f{20}{7}}(Q(R))}.$$
  This achieves the proof of this  proposition.
\end{proof}

\section{Induction arguments and proof of Theorem \ref{the1.2}}
\label{sec4}
\setcounter{section}{4}\setcounter{equation}{0}
In this section, we begin with a  critical proposition, which can be seen   as the bridge between the previous step and the next step for the given statement in the
induction arguments. Ultimately, we finish the proof of Theorem \ref{the1.2}.
\begin{prop}\label{keyinindu}
Assume that $\iint_{Q(r)}  |v|^{\f{10}{3}}\leq r^{5}N$ with
$r_{k}\leq r\leq r_{k_{0}}$.
There is a constant $C$ such that the following result holds. For any given $(x_{0},\,t_{0})\in\mathbb{R}^{n}\times \mathbb{R}^{-}$ and $k_{0}\in\mathbb{N}$, we have
for any $k>k_{0}$,
 \begin{align}\ba
&\sup_{-r_{k}^{2}\leq t-t_0\leq 0}\fbxo |v|^{2}
+r_{k}^{-3}\iint_{\tilde{Q} _{k}}
 |\nabla v |^{2}\nonumber\\
\leq&  C\sup_{-r_{k_{0}}^{2}\leq t-t_0\leq 0}\fbxozero|v|^{2}+C\sum^{k}_{l=k_{0}} r_{l}\Big(\fqxol |v|^{\f{10}{3}}   \Big)^{\f{9}{10}}
+C\sum^{k}_{l=k_{0}}r^{ \f{3}{10}}_{l} \B(\fqxol  |v|^{\f{10}{3} } \Big)^{\f{3}{5}}
\Big(\iint_{Q_{1} } |u|^{\f{20}{7}}   \Big)^{\f{7}{20}}\nonumber\\&+C\sum^{k}_{l=k_{0}}r^{ \f{13}{10}}_{l} \B(\fqxol  |v|^{\f{10}{3} } \Big)^{\f{3}{5}}
\Big(\iint_{Q_{1} } |u|^{\f{20}{7}}   \Big)^{\f{7}{20}} +C\sum^{k}_{l=k_{0}}r^{  \f35}_{l} \B(\fqxol  |v|^{\f{10}{3} } \Big)^{\f{3}{10}}
\Big(\iint_{Q_{1} } |u|^{\f{20}{7}}   \Big)^{\f{7}{10}}\nonumber\\&
+C\sum^{k }_{l=k_{0}} r_{l}\B(\fqxol  |v|^{\f{10}{3} } \Big)^{\f{3}{10}}
\Big(\iint_{\tilde{Q}_{1} } |\nabla u|^{2 }   \Big)^{\f{1}{2}}
+C
\sum^{k}_{l=k_{0}}r^{\f32}_{l}\B(\fqxol  |v|^{\f{10}{3} } \Big)^{\f{3}{10}}\B\{N^{3/5}\nonumber\\&+
 N^{\f{3}{10}} \B(\iint_{Q_{1}}|u|^{\f{20}{7}}\B)^{\f{7}{20}} + \B(\iint_{Q_{1}}|u|^{20/7}\B)^{\f{7}{10}} \B\}.\label{eq3.1}\ea
 \end{align}
\end{prop}
\begin{proof}
Without loss of generality, we suppose $(x_{0},t_{0})=(0,0)$.
We denote the backward heat kernel
$$
\Gamma(x,t)=\frac{1}{4\pi(r_{k}^2-t)^{3/2}}e^{-\frac{|x|^2}
{4(r_{k}^2-t)}}.
$$
In addition, consider the smooth cut-off functions below
$$\phi (x,t)=\left\{\ba
&1,\,~~(x,t)\in Q(r_{k_{0}+1}),\\
&0,\,~~(x,t)\in Q^{c}(\f{3}{2}r_{k_{0}+1});
\ea\right.
$$
 satisfying
$$
0\leq \phi ,\,\phi_{2}\leq1~~\text{and}~~~~r^{2}_{k_{0}}|\partial_{t}\phi  (x,t)|
+r^{l}_{k_{0}}|\partial^{l}_{x}\phi (x,t)|\leq C.
$$
To proceed further, we list some properties of the test function $\phi(x,t)\Gamma(x,t)$, whose deduction rests on elementary calculations.
\begin{enumerate}[(i)]
\item There is a constant $c>0$ independent of $r_{k}$ such that, for any $(x,t)\in Q(r_{k})$,
$$
 \Gamma(x,t)\geq c r_{k}^{-3}.
$$
\item For any $(x,t)\in Q( r_{k_{0}})$, we have
$$
|\Gamma(x,t)\phi(x,t)| \leq C r_{k}^{-3},~~~~~|\nabla\phi(x,t)\Gamma(x,t)| \leq C r_{k}^{-4}, ~~~~~|\phi(x,t) \nabla\Gamma(x,t)|\leq C r_{k}^{-4}. $$
\item For any $(x,t)\in Q(3r_{k_{0}}/4)\backslash Q(r_{k_{0}}/2)$, one can deduce that
     $$\Gamma(x,t)\leq Cr_{k_{0}}^{-3},~\partial_{i}\Gamma(x,t)\leq Cr_{k_{0}}^{-4},$$ which yields that
$$
|\Gamma(x,t)\partial_{t}\psi(x,t)|+|\Gamma(x,t) \Delta\psi(x,t)|+|\nabla\psi(x,t)\nabla\Gamma(x,t)| \leq Cr_{k_{0}}^{-5}.
$$
\item \label{property2} For any $(x,t)\in Q_{l} \backslash Q_{l+1} $,
    $$
    \Gamma\leq  C r_{l+1}^{-3},~\nabla\Gamma \leq C  r_{l+1}^{-4}.
    $$
\end{enumerate}

Now, setting $\varphi_{1}=\phi \Gamma$  in the  local energy inequality \eqref{wloc1} and utilizing the fact that $\Gamma_{t}+\Delta\Gamma=0$, we see that
  \be\ba
  &\int_{B_{1}}|v|^2\phi (x,t)\Gamma  + \int^{t}_{-r^{2}_{k_{0}}}\int_{B_{1}}\big|\nabla v\big|^2\phi (x,s)\Gamma  \nonumber\\  \leq&   \int^{t}_{-r^{2}_{k_{0}}}\int_{B_{1}} | v|^2( \Gamma\Delta \phi +\Gamma\partial_{t}\phi
  +2\nabla\Gamma\nabla\phi )   \nonumber\\&+\int^{t}_{-r^{2}_{k_{0}}}\int_{B_{1}}|v|^{2}v\cdot\nabla(\phi  \Gamma)-|v|^{2}\nabla\Pi_{h} \cdot\nabla\phi \nonumber\\
& +\int^{t}_{-r^{2}_{k_{0}}}\int_{B_{1}}\Gamma\phi ( v\otimes v-v\otimes\nabla\Pi_{h}:\nabla^{2}\Pi_{h} )  +\int^{t}_{-r^{2}_{k_{0}}}\int_{B_{1}} \Pi_{1}v\cdot\nabla(\Gamma\phi )  +\int^{t}_{-r^{2}_{k_{0}}}\int_{B_{1}}  \Pi_{2}v\cdot\nabla(\Gamma\phi )  \label{loc21}
\ea\ee
where
$$
\nabla\Pi_{1}=\mathcal{W}_{2,B_{1}}(\Delta u),~~~\nabla\Pi_{2}=-\mathcal{W}_{\f{20}{7},B_{1}}(\nabla\cdot(u\otimes u) ).
$$

First, we present the low bound estimates of the terms on the left hand side of this inequality.
Indeed, with the help of \eqref{property2}, we find
 $$\int_{B_{k}} |v|^{2} \phi  \Gamma\geq C{\fbx} |v|^{2},$$
  and
$$
\int^{t}_{-r^{2}_{k_{0}}}\int_{B_{1}}\phi  \Gamma
 |\nabla v |^{2} \geq r_{k}^{-n}\iint_{Q _{k}}
 |\nabla v |^{2} .$$
Having observed that the support of  $\partial_{t}\phi $
is included in
$\tilde{Q}(\f{3r_{k_{0}}}{4})/\tilde{Q}(\f{r_{k_{0}}}{2}),$
 we get
$$\int^{t}_{-r^{2}_{k_{0}}}\int_{B_1}
 |v|^{2}
\Big|\Gamma\Delta \phi +\Gamma\partial_{t}\phi +2\nabla\Gamma\nabla\phi \Big|
\leq C\sup_{-r_{k_{0}}^{2}\leq t\leq0}\fbxozeroo |v|^{2} .
$$
  H\"older's inequality and \eqref{property2} enable us to write  that
\begin{align}
&\iint_{Q_{k_{0}}}|v|^{2}v\cdot\nabla(\phi  \Gamma) d  \tau \nonumber\\
\leq&\sum^{k-1}_{l=k_{0}}\iint_{Q_{l}/Q_{l+1}}
  |v|^{3}  |\nabla(\phi  \Gamma)|  +\iint_{Q_{k}}  |v|^{3} |\nabla(\phi  \Gamma)|   \nonumber\\
\leq&\sum^{k}_{l=k_{0}}r_{l}^{-4}\iint_{Q_{l}}
   |v|^{3}  \nonumber\\
 \leq& C\sum^{k}_{l=k_{0}} r_{l}\Big(\fqxolo |v|^{\f{10}{3}}   \Big)^{\f{9}{10}}\nonumber.
\end{align}
Following the lines of reasoning which led to the last inequality,   we have
\begin{align}
&\iint_{Q_{k_{0}}}|v|^{2}\nabla\Pi_{h} \cdot\nabla(\phi\Gamma)  \nonumber\\
\leq&\sum^{k}_{l=k_{0}}r_{l}^{-4}\iint_{Q_{l}}
 |v|^{2}| \nabla\Pi_{h}|  \nonumber\\
\leq& C\sum^{k}_{l=k_{0}}r^{ \f{3}{10}}_{l} \B(\fqxolo  |v|^{\f{10}{3} } \Big)^{\f{3}{5}}
\Big(\iint_{Q_{1} } |u|^{\f{20}{7}}   \Big)^{\f{7}{20}}.
\end{align}
Likewise, we have
\begin{align}
&\iint_{Q_{k_{0}}}|v|^{2}|\nabla^{2} \Pi_{h}| (\phi\Gamma) \nonumber\\
\leq&\sum^{k}_{l=k_{0}}r_{l}^{-3}\iint_{Q_{l}}
 |v|^{2}| \nabla^{2} \Pi_{h}|  \nonumber\\
\leq& C\sum^{k}_{l=k_{0}}r^{ \f{13}{10}}_{l} \B(\fqxolo  |v|^{\f{10}{3} } \Big)^{\f{3}{5}}
\Big(\iint_{Q_{1} } |u|^{\f{20}{7}}   \Big)^{\f{7}{20}}\nonumber.
\end{align}
Using H\"older's inequality  again, \eqref{property2}, \eqref{h1} and \eqref{p1}, we infer that
\begin{align}
 &\iint_{Q_{k_{0}}}\phi  \Gamma   |v| |\nabla\Pi_{h}||\nabla^{2}\Pi_{h}|\nonumber\\
\leq&  C\sum^{k}_{l=k_{0}}r_{l}^{-n}
\Big(\iint_{Q_{l} } |v|^{\f{10}{3} }   \Big)^{\f{3}{10}}
\Big(\iint_{Q_{l} } |\nabla\Pi_{h}|^{\f{20}{7} }   \Big)^{\f{7}{20}}\Big(\iint_{Q_{l} } |\nabla^{2} \Pi_{h}|^{\f{20}{7} }   \Big)^{\f{7}{20}}\nonumber\\
\leq& C\sum^{k}_{l=k_{0}}r^{  \f35}_{l} \B(\fqxolo  |v|^{\f{10}{3} } \Big)^{\f{3}{10}}
\Big(\iint_{Q_{1} } |u|^{\f{20}{7}}   \Big)^{\f{7}{10}}\nonumber.
\end{align}
Set $\chi_{l}=1$ on $|x|\leq7/8 r_{l}$ and $\chi_{l}=0 $ if $ |x|\geq r_{l}$.
$\chi_{k_{0}}\Gamma=\Gamma$ on $Q_{k_{0}}$
By the support of $(\chi_{l}-\chi_{l+1})$, we derive from \eqref{property2}  that $|\nabla((\chi_{l}-\chi_{l+1})\phi  \Gamma)|\leq Cr^{-4}_{l+1}$.
With the help  of  \eqref{property2}  again, we see that $|\nabla( \chi_{k} \phi  \Gamma)|\leq Cr^{-4}_{k} $. Therefore, it holds
\begin{align}
\iint_{Q_{k_{0}}}v\cdot\nabla(\phi  \Gamma) \Pi_{1}=&
\sum^{k-1}_{l=k_{0}}\iint_{Q_{l}}v\cdot\nabla((\chi_{l}-\chi_{l+1})\phi  \Gamma) \Pi_{1}+\iint_{Q_{k}}v\cdot\nabla( \chi_{k} \phi  \Gamma) \Pi_{1}\nonumber\\
=&
\sum^{k-1}_{l=k_{0}}\iint_{Q_{l}}
v\cdot\nabla((\chi_{l}-\chi_{l+1})\phi  \Gamma) (\Pi_{1}- \overline{{\Pi_{1}}}_{l})+\iint_{Q_{k }}u\cdot\nabla( \chi_{k} \phi  \Gamma) (\Pi_{1}-\overline{{\Pi_{1}}}_{k})
\nonumber\\
\leq& C
\sum^{k-1}_{l=k_{0}}r^{-4}_{l+1}\iint_{Q_{l}}
|v||\Pi_{1}- \overline{{\Pi_{1}}}_{l}|+r^{-4}_{k}\iint_{Q_{k }}|v||\Pi_{1}-\bar{\Pi_{1}}_{k}|\nonumber\\
=&:I+II.
\end{align}
The H\"older inequality, \eqref{h2} and \eqref{p2} give
\begin{align}
I\leq& C
\sum^{k-1}_{l=k_{0}}r^{-4}_{l+1}
\Big(\iint_{\tilde{Q}_{l} } |v|^{\f{10}{3} }   \Big)^{\f{3}{10}}
\Big(\iint_{\tilde{Q}_{l} } |\Pi_{1}- \overline{{\Pi_{1}}}_{l}|^{2 }   \Big)^{\f{1}{2}} r_{l}\nonumber\\\leq& C
\sum^{k-1}_{l=k_{0}}r^{-\f52}_{l+1}
 \B(\fqxolo  |v|^{\f{10}{3} } \Big)^{\f{3}{10}} r_{l}^{ \f{7}{2} }
\Big(\iint_{\tilde{Q}_{1} } |\Pi_{1}- \overline{{\Pi_{1}}}_{1}|^{2 }   \Big)^{\f{1}{2}}  \nonumber\\
\leq& C
\sum^{k-1}_{l=k_{0}}r_{l+1}
 \B(\fqxolo  |v|^{\f{10}{3} } \Big)^{\f{3}{10}}
\Big(\iint_{\tilde{Q}_{1} } |\Pi_{1}|^{2 }   \Big)^{\f{1}{2}}\nonumber\\
\leq& C
\sum^{k-1}_{l=k_{0}}r_{l+1}
 \B(\fqxolo  |v|^{\f{10}{3} } \Big)^{\f{3}{10}}
\Big(\iint_{\tilde{Q}_{1} } |\nabla u|^{2 }   \Big)^{\f{1}{2}},
\end{align}
and
$$\ba
II
\leq& C
 r_{k}
 \B(\fqxolo  |v|^{\f{10}{3} } \Big)^{\f{3}{10}}
\Big(\iint_{\tilde{Q}_{1} } |\nabla u|^{2 }   \Big)^{\f{1}{2}},
\ea$$
which turns out that
$$\iint_{\tilde{Q}_{k_{0}}}v\cdot\nabla(\phi  \Gamma) \Pi_{1}\leq C\sum^{k }_{l=k_{0}} r_{l}\B(\fqxolo  |v|^{\f{10}{3} } \Big)^{\f{3}{10}}
\Big(\iint_{\tilde{Q}_{1} } |\nabla u|^{2 }   \Big)^{\f{1}{2}}.
$$
Note that
\be\label{4.7}u\otimes u=v\otimes v-v\otimes\nabla\Pi_{h}-\nabla\Pi_{h}\otimes v+\nabla\Pi_{h}\otimes \nabla\Pi_{h}.\ee
For $r_{k}\leq r\leq r_{k_{0}}$, we compute directly that
\be\label{4.8}
\iint_{Q(r)}
 |v\otimes v- \overline{ {v\otimes v }}_{l}|^{10/7} \leq \iint_{Q(r)}
 |v|^{20/7}\leq  Cr^{5}\B(\fqxolor  |v|^{\f{10}{3} } \Big)^{6/7}\leq   Cr^{5}N^{6/7}.
\ee
The H\"older inequality and \eqref{h1} ensure  that
\begin{align}
\iint_{Q(r)}
 |v\otimes\nabla\Pi_{h}- \overline{v\otimes\nabla\Pi_{h}}_{l}|^{10/7} &\leq C\iint_{Q(r)}
 |v\otimes\nabla\Pi_{h}|^{10/7}\nonumber\\
&\leq C \B(\iint|v|^{\f{10}{3}}\B)^{\f37}\B(\iint|\nabla\Pi_{h}|^{\f{20}{7}}\B)^{\f12}r^{\f{5}{14}}
\nonumber\\
&\leq C r^{4}\B(\fqxolor|v|^{\f{10}{3}}\B)^{\f37} \B(\iint_{Q_{1}}|u|^{\f{20}{7}}\B)^{\f12}\nonumber\\
&\leq C r^{4}N^{\f37} \B(\iint_{Q_{1}}|u|^{\f{20}{7}}\B)^{\f12}.\label{4.9}
\end{align}
In view of Poincar\'e inequality for a ball, H\"older's inequality, \eqref{h1} and \eqref{p1}, we arrive at
\begin{align}
&\iint_{Q(r)}
 |\nabla\Pi_{h}\otimes \nabla\Pi_{h}- \overline{\nabla\Pi_{h}\otimes \nabla\Pi_{h}}_{l}|^{10/7}\nonumber\\
\leq&
Cr^{ 10/7}\B(\iint_{Q(r)}|\nabla\Pi_{h}|^{20/7}\B)^{\f12}\B(\iint_{Q(r)}| \nabla^{2} \Pi_{h}|^{20/7}\B)^{\f12}\nonumber\\
\leq&
Cr^{ \f{31}{7}}\B(\iint_{Q_{1}}|\nabla\Pi_{h}|^{20/7}\B)^{\f12}\B(\iint_{Q_{1}}| \nabla\Pi_{h}|^{20/7}\B)^{\f12}\nonumber\\
\leq&
Cr^{ \f{31}{7}}\B(\iint_{Q_{1}}|u|^{20/7}\B).\label{4.10}
\end{align}
We deduce from
\eqref{4.7}-\eqref{4.10} that
\begin{align}
\iint_{Q(r)}
 |u\otimes u- \overline{ (u\otimes u )_{r}} |^{10/7} \leq &Cr^{5} N^{6/7}+
r^{4} N^{\f37} \B(\iint_{Q_{1}}|u|^{\f{20}{7}}\B)^{\f12} +Cr^{\f{31}{7}}\B(\iint_{Q_{1}}|u|^{20/7}\B) \nonumber\\
 \leq &Cr^{4} \B\{N^{6/7}+
  N^{\f37} \B(\iint_{Q_{1}}|u|^{\f{20}{7}}\B)^{\f12} +C \B(\iint_{Q_{1}}|u|^{20/7}\B)\B\}.\label{4.13}\end{align}
 With
   \eqref{4.13} in hand, we can apply
Lemma \ref{CW} to obtain that
\be\label{4.12}\ba
\iint_{Q(r)}
 |\Pi_{2}- \overline{ \Pi_{2}}_{B(r)}|^{10/7} &\leq
 Cr^{4} \B\{N^{6/7}+
  N^{\f37} \B(\iint_{Q_{1}}|u|^{\f{20}{7}}\B)^{\f12} +C \B(\iint_{Q_{1}}|u|^{20/7}\B)\B\}.
 \ea\ee
 Particulary, for any $k\leq  l\leq k_{0}$, it holds
 \be\label{4.20}\ba
\iint_{Q_{l}}
 |\Pi_{2}- \overline{ \Pi_{2}}_{B_{l}}|^{10/7} &\leq
 Cr_{l}^{4} \B\{N^{6/7}+
  N^{\f37} \B(\iint_{Q_{1}}|u|^{\f{20}{7}}\B)^{\f12} +C \B(\iint_{Q_{1}}|u|^{20/7}\B)\B\}.
 \ea\ee
By the H\"older inequality, we see that
\be\label{4.11}
r_{l}^{-4}\iint_{Q_{l}}|v||\Pi_{2}-\overline{ \Pi_{2}}_{B(r)}|\leq
Cr_{l}^{-4}\Big(\iint_{\tilde{Q}_{l} } |v|^{\f{10}{3} }   \Big)^{\f{3}{10}}
\Big(\iint_{\tilde{Q}_{l} } |\Pi_{2}{ -\overline{(\Pi_{2})}_{B_{l}}}|^{\f{10}{7} }   \Big)^{\f{7}{10}}.
\ee
Plugging \eqref{4.12} into \eqref{4.11}, we have
$$\ba
&\iint_{Q_{k_{0}}}v\cdot\nabla(\phi  \Gamma) \Pi_{2}\\
\leq& C
\sum^{k-1}_{l=k_{0}}r^{-4}_{l+1}\iint_{Q_{l}}
|v||\Pi_{2}- \overline{{\Pi_{2}}}_{l}|+r^{-4}_{k}\iint_{Q_{k }}|v||\Pi_{2}-\overline{ {\Pi_{2}}}_{k}|\\
\leq& C
\sum^{k}_{l=k_{0}}r^{\f{3}{10}}_{l}\B(\fqxolo  |v|^{\f{10}{3} } \Big)^{\f{3}{10}} \B\{N^{6/7}+
  N^{\f37} \B(\iint_{Q_{1}}|u|^{\f{20}{7}}\B)^{\f12} +C \B(\iint_{Q_{1}}|u|^{20/7}\B)\B\}^{\f{7}{10}}.
  \ea$$

Finally, collected these estimates  leads  to \eqref{eq3.1}.
\end{proof}

With Proposition \ref{keyinindu} at our disposal, we will now present the proof of
  Theorem \ref{the1.1}.
\begin{proof}[Proof of Theorem \ref{the1.2}]
By the  interior estimate \eqref{h1}  of harmonic function   and \eqref{wp1}, we have
 \be\| \nabla \Pi_{h}\|_{L^{\infty}
(\tilde{B}(1/8))}\leq C\|\nabla \Pi_{h}\|_{L^{20/7 }(B(1))}\leq   C\|u\|_{L^{20/7 }(B(1))}.\label{last1}\ee
Assume for a while we have proved that, for any Lebesgue point $(x_{0},t_{0})\in Q(1/8)$,
\be\label{temp}
|v(x_{0},t_{0})|\leq C.
\ee
We derive from \eqref{last1} and \eqref{temp} that
$$
\| u\|_{L^{20/7,\infty}
(\tilde{Q}(1/8))}\leq \| \nabla \Pi_{h}\|_{L^{20/7,\infty}(\tilde{Q}(1/8))}+
\| v\|_{L^{20/7,\infty}(\tilde{Q}(1/8))}\leq C\|u\|_{L^{20/7 }(Q(1))}.
$$
By the well-known Serrin  regularity criteria in \cite{[Serrin]}, we know  that $(0, 0)$ is a regular point.
Therefore, it remains to prove \eqref{temp}.
In what follows, let   $(x_{0},t_{0})\in Q(1/8)$ and $r_{k}=2^{-k}$.
According to the Lebesgue differentiation theorem, it  suffices to show
\be\label{goal}
 \fqxoo |v|^{\f{10}{3}}
\leq \varepsilon_{1}^{2/3}, ~~k\geq3.\ee
First, we show that \eqref{goal} is valid for $k=3$.  Indeed, from \eqref{keyl} in section  \ref{sec3}, Proposition \ref{lebp} and hypothesis \ref{jww}, we know that
 \begin{align}
  &\sup_{-(\f{3}{8})^{2}\leq t\leq0}\int_{B(\f{3}{8})}|v |^2   d  x+ \iint_{Q(\f{3}{8})}\big|\nabla v\big|^2  d  x d  \tau\nonumber\\\leq&   C
 \| u \|^{4}_{L^{20/7}(Q(\f{1}{2}))}
 +C\| u \|^{2}_{L^{20/7}(Q(\f{1}{2}))}+\f{1}{16}\| \nabla u\|^{2}_{L^{2}(Q(\f{1}{2}))},
  \nonumber\\\leq&      C\varepsilon^{^{7/10}}. \end{align}
In light of Sobolev embeddings  and the Young inequality, we see that
\be\label{eq3.5}\ba
 \B(\iint_{Q(\f{3}{8})}|v |^{\f{10 }{3}}\B)^{\f{3}{10 }}
&\leq C\B(\sup_{-(\f{3}{8})^{2}\leq t<0}\int_{B(\f{3}{8})}|v|^{2}\B)^{1/2}+ C\B(\iint_{Q (\f{3}{8})} |\nabla  v |^{2}\Big)^{1/2}.
\ea\ee
 It turns out that
$$
   \fqxoth |v|^{\f{10}{3} }\leq C\varepsilon_{1}^{\f{7}{6}}.$$
This proves \eqref{goal} in the case $k=3$.
Now, we assume that, for any $3\leq l\leq k$,
$$
  \fqxol |v|^{\f{10}{3} }
\leq \varepsilon_{1}^{2/3}.$$
Furthermore, there holds, for any $r_{k}\leq r\leq r_{3}$
$$
  \fqxolorr |v|^{\f{10}{3} }
\leq C\varepsilon_{1}^{2/3}.$$
For any $3\leq i\leq k$, by Proposition \ref{keyinindu} with $N=C\varepsilon_{1}^{2/3}$, \eqref{jww} and the above  induction hypothesis, we find that
\begin{align}
&\sup_{-r_{k}^{2}\leq t-t_0\leq 0}\fbxo |v|^{2}
+r_{k}^{-n}\iint_{\tilde{Q} _{k}}
 |\nabla v |^{2}\nonumber\\
\leq&  C\sup_{-r_{3}^{2}\leq t-t_0\leq 0}\fbxozero|v|^{2}+C\sum^{k}_{l=k_{0}} r_{l}\Big(\fqxolo |v|^{\f{10}{3}}   \Big)^{\f{9}{10}} \nonumber\\&
+C\sum^{k}_{l=3}r^{ \f{3}{10}}_{l} \B(\fqxolo  |v|^{\f{10}{3} } \Big)^{\f{3}{5}}
\Big(\iint_{Q_{1} } |u|^{\f{20}{7}}   \Big)^{\f{7}{20}}\nonumber\\&+C\sum^{k}_{l=3}r^{ \f{13}{10}}_{l} \B(\fqxolo  |v|^{\f{10}{3} } \Big)^{\f{3}{5}}
\Big(\iint_{Q_{1} } |u|^{\f{20}{7}}   \Big)^{\f{7}{20}}\nonumber\\& +C\sum^{k}_{l=3}r^{  \f35}_{l} \B(\fqxolo  |v|^{\f{10}{3} } \Big)^{\f{3}{10}}
\Big(\iint_{Q_{1} } |u|^{\f{20}{7}}   \Big)^{\f{7}{10}}\nonumber\\&
+C\sum^{k }_{l=3} r_{l}\B(\fqxolo  |v|^{\f{10}{3} } \Big)^{\f{3}{10}}
\Big(\iint_{\tilde{Q}_{1} } |\nabla u|^{2 }   \Big)^{\f{1}{2}}\nonumber\\&
+C
\sum^{k}_{l=3}r^{\f{3}{10}}_{l}\B(\fqxolo  |v|^{\f{10}{3} } \Big)^{\f{3}{10}}\B\{N^{3/5}+
 N^{\f{3}{10}} \B(\iint_{Q_{1}}|u|^{\f{20}{7}}\B)^{\f{7}{20}} + \B(\iint_{Q_{1}}|u|^{20/7}\B)^{\f{7}{10}} \B\}\nonumber\\\leq&  C\varepsilon^{7/5}+C\sum^{k}_{l=3} r_{l}\varepsilon^{\f35}
+C\sum^{k}_{l=3}r^{ \f{3}{10}}_{l} \varepsilon^{\f{2}{5}}\varepsilon^{\f{7}{20}}+
C\sum^{k}_{l=3}r^{ \f{13}{10}}_{l} \varepsilon^{\f{2}{5}}
\varepsilon^{\f{7}{20}}\nonumber\\& +C\sum^{k}_{l=3}r^{  \f35}_{l} \varepsilon^{\f{1}{5}}
\varepsilon^{\f{7}{10}}
+C\sum^{k }_{l=3} r_{l} \varepsilon^{\f{1}{5}}
\varepsilon^{\f{7}{20}}
+C
\sum^{k}_{l=3}r^{\f{3}{10}}_{l}\
\varepsilon^{\f{1}{5}}\B\{\varepsilon^{\f{2}{5}}+
\varepsilon^{\f{1}{5}} \varepsilon^{\f{7}{20}}+ \varepsilon^{\f{7}{10}} \B\}\nonumber\\
\leq& C \varepsilon^{\f{11}{20}}.\label{eq3.1}
\end{align}
Invoking   the
Gagliardo-Nirenberg inequality, we deduce that
$$\ba
\int_{\tilde{B}_{k+1} }|v|^{\f{10}{3}}dx
&\leq  C\B(\int_{\tilde{B}_{k}}|v |^{2}\B)^{\f{2}{3}}
\B[\B(\int_{\tilde{B}_{k} } |\nabla  v |^{2}\Big)^{1/2}+r_{k}^{-1}
\B(\int_{\tilde{B}_{k}}|v|^{2}\B)^{1/2}\B]^{2},
\ea$$
which means
$$\ba
\iint_{\tilde{Q}_{k+1} }|v|^{\f{10}{3}}
\leq& C\B(\sup_{-r^{2 }_{k}\leq t-t_{0}<0}\int_{\tilde{B}_{k}}|v |^{2}\B)^{\f{2}{3}}
 \B(\iint_{\tilde{Q}_{k} } |\nabla  v |^{2}\Big) +
\B(\sup_{-r^{2 }_{k}\leq t-t_{0}<0}\int_{\tilde{B}_{k}}|v|^{2}\B)^{5/3}.
\ea$$
This inequality, combined with \eqref{eq3.1}, implies that
\be\label{last2}\ba
\f{1}{r^{5 }_{k+1}}\iint_{\tilde{Q}_{k+1}}
|v|^{\f{10}{3}}\leq&  C\B( \f{1}{r^{3}_{k}}\sup_{-r^{2 }_{k}\leq t-t_{0}<0}\int_{\tilde{B}_{k}}|v |^{2}\B)^{\f53}\\&
+C\B( \f{1}{r^{3}_{k}}\sup_{-r^{2 }_{k}\leq t-t_{0}<0}\int_{\tilde{B}_{k}}|v |^{2}\B)^{\f32}\B(r_{k}^{-3}\iint_{\tilde{Q} _{k}}
|\nabla v |^{2}\B) \\
\leq& C  \varepsilon_{1}^{\f{11}{12}}.
\ea
\ee
Collecting the above bounds, we eventually conclude
that
$$\iint_{\tilde{Q}_{k+1}} \!\!\!\!\!\!\!\!\!\! \! \!\!\!\!\!\!\!\!-\hspace{-0.16cm}-\hspace{-0.15cm}~~\,~
|v|^{\f{10}{3} }\leq  \varepsilon_{1}^{2/3}. $$
This completes the proof of this theorem.
\end{proof}

\section*{Acknowledgement}
Jiu was partially supported by the National Natural Science Foundation of China (No.11671273, No.11231006).
		The research of Wang was partially supported by  the National Natural		Science Foundation of China under grant No. 11601492 and the
the Youth Core Teachers Foundation of Zhengzhou University of
Light Industry.
			The research of Zhou	is supported in part by the National Natural Science
Foundation of China under grant No. 11401176  and Doctor Fund of Henan Polytechnic University (No. B2012-110).

\end{document}